\documentclass[11pt]{article}
\usepackage{amssymb,amsmath,amsthm}
\usepackage{epsfig}
\usepackage{footnpag}
\usepackage{url}
\usepackage{array} 

\newtheorem{thm}{Theorem}

\newtheorem{prop}[thm]{Proposition}

\newtheorem{conj}[thm]{Conjecture}

\newtheorem{defi}[thm]{Definition}

\theoremstyle{remark}
\newtheorem{remark}[thm]{Remark}

\usepackage{indentfirst}
\usepackage{xspace} 

\def\vc{\mathrm{VC}}
\def\lcs{\mathrm{LCS}}
\def\ucs{\mathrm{UCS}}

\def\F{\mbox{\ensuremath{\mathcal F}}\xspace}
\def\G{\mbox{\ensuremath{\mathcal G}}\xspace}

\mathcode`l="8000
\begingroup
\makeatletter
\lccode`\~=`\l
\DeclareMathSymbol{\lsb@l}{\mathalpha}{letters}{`l}
\lowercase{\gdef~{\ifnum\the\mathgroup=\m@ne \ell \else \lsb@l \fi}}%
\endgroup

\makeatletter
\def\thickhline{%
	\noalign{\ifnum0=`}\fi\hrule \@height \thickarrayrulewidth \futurelet
	\reserved@a\@xthickhline}
\def\@xthickhline{\ifx\reserved@a\thickhline
	\vskip\doublerulesep
	\vskip-\thickarrayrulewidth
	\fi
	\ifnum0=`{\fi}}
\makeatother

\newlength{\thickarrayrulewidth}
\newcommand\restr[2]{{
  \left.\kern-\nulldelimiterspace 
  #1 
  \vphantom{\big|} 
  \right|_{#2} 
  }}

\begin{document}

\title{Unlabeled Compression Schemes Exceeding the VC-dimension}
\author{D\"om\"ot\"or P\'alv\"olgyi\footnote{MTA-ELTE Lend\"ulet Combinatorial Geometry Research Group, Institute of Mathematics, E\"otv\"os Lor\'and University (ELTE), Budapest, Hungary. Research supported by the Lend\"ulet program of the Hungarian Academy of Sciences (MTA), under grant number LP2017-19/2017.}~~and G\'abor Tardos\footnote{Supported by the Cryptography ``Lend\"ulet'' project of the Hungarian Academy of Sciences and by the National
Research, Development and Innovation Office, NKFIH projects K-116769, KKP-133864 and SNN-117879.}}

\maketitle

\begin{abstract}
In this note we disprove a conjecture of Kuzmin and Warmuth claiming that every family whose VC-dimension is at most $d$ admits an unlabeled compression scheme to a sample of size at most $d$.
We also study the unlabeled compression schemes of the joins of some families and conjecture that these give a larger gap between the VC-dimension and the size of the smallest unlabeled compression scheme for them.
\end{abstract}

\section{Introduction}

In statistical learning, it is important to derive information from a large sample space and store only the essential part of it.
The goal of this paper is to study a model of this learning process, and show that certain samples cannot be compressed optimally.

Terminology: if $S$ is a subset of the domain of a function $f$, then we call the restriction $g=f|_S$ the \emph{trace} of $f$ on $S$ and we also call $f$ an \emph{extension} of $g$.

Consider a finite set $B$, and fix a family \F of functions $B\to\{0,1\}$. For $f\in\F$ and $S\subseteq B$
we call the trace $f|_S$ a \emph{partial function} of the family $\F$.
These are studied extensively in learning theory, where our goal is to reconstruct $f|_S$ from some part of it.

\begin{defi}[Littlestone and Warmuth \cite{LW}]
	A \emph{(labeled) compression scheme} for a family $\cal F$ of binary functions with domain $B$ is a pair of operations $(\alpha,\beta)$ such that
		\begin{itemize}
		\item $\alpha$ takes a partial function $g$ of $\F$ as an input (called a labeled sample) and returns a  trace of $g$,
		\item $\beta$ takes the output of $\alpha$ as input and returns an arbitrary function $f:B\to\{0,1\}$,
		\item $\beta(\alpha(g))$ is an extension of $g$ for any partial function $g$ of \F.
	\end{itemize}
\end{defi}

That is, instead of $f|_S$, it is enough to store $\alpha(f|_S)$ so that we can fully recover the value of $f$ over $S$.
The size of the compression scheme $(\alpha,\beta)$ is the maximum size of the domain of $\alpha(g)$.
We denote by $\lcs(\F)$ the minimum size of a compression scheme for $\F$.

\begin{remark}
	Note that the domain $S$ of the partial function $g$ is not required to be determined by the sample $\alpha(g)$, and $S$ is not given the reconstruction process $\beta$ when producing $\beta(\alpha(g))$.
\end{remark}
\begin{remark}
	$\beta(\alpha(f|_S))$ is not required to be from $\F$.
\end{remark}

\begin{defi}[Vapnik-Chervonenkis \cite{VC}]
Let \F be a family of functions $B\to\{0,1\}$.
We say that \F shatters $X \subseteq B$ if every function $g:X\to\{0,1\}$ has an extension in \F.
The \emph{VC-dimension} of \F, $\vc(\F)$, is defined as
the size of the largest $X$ that is shattered by \F.
\end{defi}

Littlestone and Warmuth \cite{LW} observed that $\lcs(\F)\ge \vc(\F)/5$ always holds but could not give any compression scheme for general families whose size depended only on $\vc(\F)$.
Floyd and Warmuth \cite{FW} conjectured that $\lcs(\F)\le \vc(\F)$ always holds.
(There are simple examples that show that this would be sharp.)
Warmuth \cite{W} even offered \$600 reward for a proof that a compression scheme of size $O(d)$ always exists, but this has been proved only in special cases.
Most notably, Floyd and Warmuth \cite{FW} claimed to have proved it for families of VC-dimension $d$ whose size is $\sum_{i=0}^d \binom{n}{i}$, i.e., the maximum size allowed by the Sauer-Shelah lemma; an error in the original argument was recently fixed in \cite{CCMW}.

In 2015, Moran and Yehudayoff \cite{MY} have managed to prove that a compression scheme exists whose size depends only on $\vc(\F)$, but their bound is exponential in $\vc(\F)$.

\begin{defi}[Kuzmin and Warmuth \cite{KW}]
		An \emph{unlabeled compression scheme} for a family $\cal F$ of binary functions with domain $B$ is a pair of operations $(\alpha,\beta)$ such that
		\begin{itemize}
			\item $\alpha$ takes a partial function $g$ of $\cal F$ with domain $S$ (called a labeled sample) and returns a $\alpha(g)$ (called the compressed sample), which is a subset of $S$,
			\item $\beta$ takes the output of $\alpha$ as input and returns an arbitrary function $f:B\to\{0,1\}$,
			\item $\beta(\alpha(g))$ is an extension of $g$ for any partial function $g$ of \F.
		\end{itemize}
\end{defi}

That is, unlike in the case of labeled compression schemes, we do not store the value of $f$ on the compressed sample, but only some selected sample points.
The size of the unlabeled compression scheme $(\alpha,\beta)$ is the maximum size of $\alpha(g)$ for any partial function $g$.
We denote by $\ucs(\F)$ the minimum size of an unlabeled compression scheme for $\F$.
Note that $\ucs(\F)\ge \lcs(\F)$ trivially holds.

Kuzmin and Warmuth \cite{KW} have proved that $\ucs(\F)\ge \vc(\F)$ and conjectured that equality might hold for every family (a strengthening of the earlier conjecture of Floyd and Warmuth).
Similarly to the labeled case, they also claimed a proof for maximum size families; this seems to have contained a similar error, and was also fixed in \cite{CCMW}.

We disprove this last conjecture in a very weak sense;
we exhibit a small family $C_5$ for which $\vc(C_5)=2$ but $\ucs(C_5)=3$.
We also discuss possible ways to amplify this gap, but at the moment we do not know any family \F with $\ucs(\F)> \vc(\F)$ for which $\ucs(\F)\ge 4$, although we exhibit some likely candidates.

\section{Lower bound for $C_5$}\label{sec:lower}

Here we define the family $C_5$ for which $\ucs(C_5)=3>\vc(C_5)=2$, and prove these equalities.
The domain of $C_5$ is five elements and $|C_5|=10$; see Figure \ref{fig:C5}.
We think of the domain $B$ of $C_5$ as the vertices of a regular pentagon. A 0-1 function on this domain belongs to $C_5$ if and only if it takes the values 1-0-0-1 on some four consecutive vertices.

\begin{figure}[h]
	\begin{center}
		\includegraphics[width=5cm]{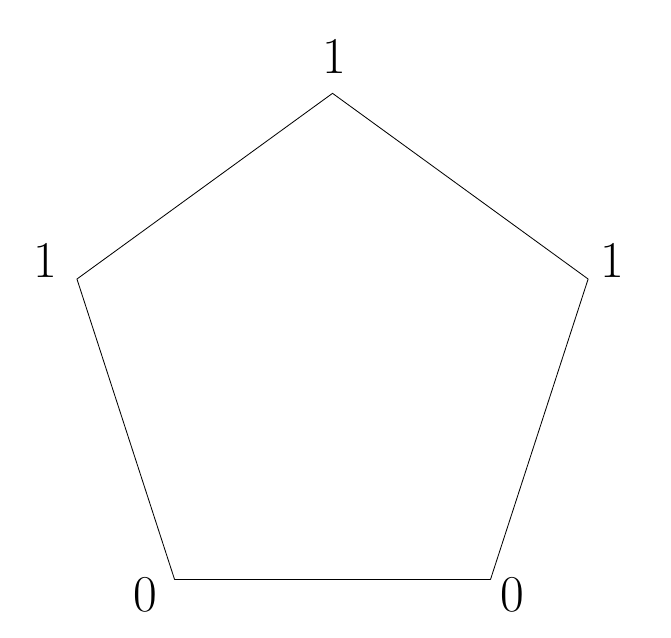}~~~~~~~~~~~
		\includegraphics[width=5cm]{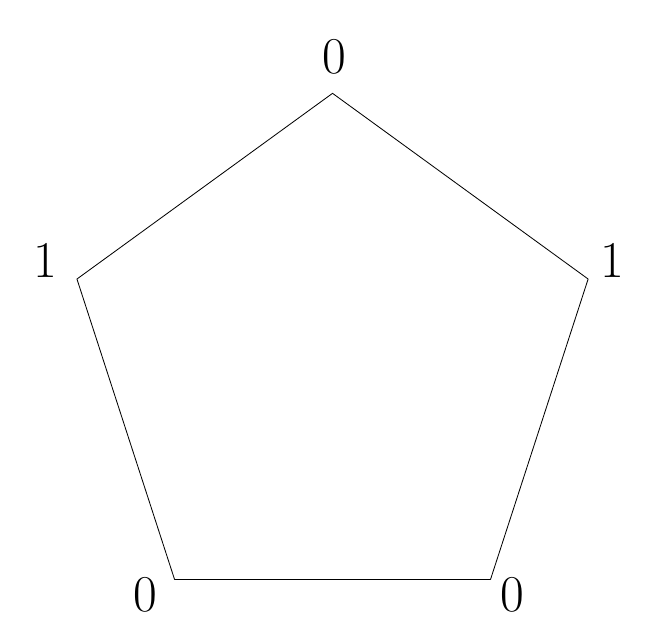}%
	\end{center}
	\caption{$C_5$ consists of the $5$ rotations of the above sets.}\label{fig:C5}
\end{figure}

As we have later found out, this is known in the learning theory literature as `Warmuth's example.' He constructed it as a simple example of a containment maximal family with $\vc(C_5)=2$ that does not reach the maximal size of such a family given by the Sauer-Shelah lemma, which in this case would be  $\sum_{i=0}^2 \binom{5}{i}=16$. 


We will use the property that for any subset $S\subset B$ of size $3$ there are $7$ possibilities for the trace $f|_S$ for $f\in C_5$.
If $S$ consists of three consecutive vertices, then $f|_S$ cannot be constant $0$, while if $S$ consists of three non-consecutive vertices the constant $1$ trace is not possible.
Note that this implies that $C_5$ shatters no three element set but it shatters all two element sets, so its VC-dimension is $2$.

We identify the domain $B$ of $C_5$ with the residue classes modulo 5, with the neighbors of the vertex $i\in B$ being $i+1$ and $i-1$.

\begin{thm}\label{thm}
	$\ucs(C_5)=3$.
\end{thm}
\begin{proof}
	It is easy to construct an unlabeled compression scheme of size $3$: $\alpha$ can keep the sample points where the value of the function is $1$, and the reconstruction function $\beta$ returns $1$ at every place contained in the compressed sample, and $0$ everywhere else.
	Thus, we only need to prove that $\ucs(C_5)\ge 3$.
	
	Suppose by contradiction that there is an unlabeled compression scheme $(\alpha,\beta)$ of size two.
	Let $X$ be a size $3$ subset of the domain. As we noted above, there are exactly $7$ partial functions $g:X\to\{0,1\}$ of $C_5$. Clearly, $\alpha(g)$ must be a distinct proper subset of $X$ for each.
	As there are $7$ such subsets, we must have a $1$-$1$ correspondence here.
	In particular, for all $Y\subsetneq X$, the $\beta(Y)|_X$ must be distinct partial functions of $C_5$.
	
	Let $J$ be the set of three consecutive positions in the domain and $i\in J$. Let $g$ be the constant $0$ partial function defined on $J\setminus\{i\}$ and $Y=\alpha(g)$. Here $\beta(Y)|_J$ is a partial function of $C_5$ extending $g$, so it must be $0$ on $J\setminus\{i\}$ and $1$ on $i$. Now $\beta(\{i\})|_J$ must be another partial function of $C_5$, therefore $\beta(\{i\})|_{(J\setminus\{i\})}$ cannot be constant $0$. A symmetric argument shows that if $K$ is the set of three non-consecutive positions and $i\in K$, then $\beta(\{i\})|_{(K\setminus\{i\})}$ is not constant $1$.
	
	The observations above imply that $\beta(\{i\})(i-1)=1$. Indeed, if $\beta(\{i\})(i-1)=0$, then applying the observation in the previous paragraph for $J=\{i-2,i-1,i\}$ we obtain $\beta(\{i\})(i-2)=1$ and considering $J=\{i-1,i,i+1\}$ we obtain $\beta(\{i\})(i+1)=1$, but this contradicts our observation about $K=\{i-2,i,i+1\}$. A similar argument shows $\beta(\{i\})(i+1)=1$ as well as $\beta(\{i\})(i-2)=\beta(\{i\})(i+2)=0$. The only remaining value, namely $\beta(\{i\})(i)$ therefore completely determines $\beta(\{i\})$.
	
	Suppose $\beta(\{i\})(i)=1$ holds for at least three different values of $i$; then it must hold for two consecutive values, say $i$ and $i+1$. This completely determines $\beta(\{i\})$ and $\beta(\{i+1\})$ and these functions coincide on $X=\{i-2,i,i+1\}$ contradicting our observation that for distinct proper subsets $Y$ of $X$, the $\beta(Y)|_X$ must also be distinct.
	
	Alternatively we must have $\beta(\{i\})(i)=0$ for at least three different values of $i$. Then it also holds for	two non-consecutive values, say $i-1$ and $i+1$. This completely determines $\beta(\{i-1\})$ and $\beta(\{i+1\})$ and these functions coincide on $X=\{i-1,i,i+1\}$, a contradiction again. The contradictions prove the theorem.	
\end{proof}

\begin{remark}
	Note that the above proof in fact shows that in any unlabeled compression scheme for $C_5$ there is a partial function with domain of size $3$ that gets compressed to its full domain.
\end{remark}

\section{Upper bounds for $C_5$'s}

In this section we sketch some upper bounds, i.e., give unlabeled compression schemes for certain families.
When we receive a sample $f|_S$, we interpret it as receiving a collection of 0's and 1's, and we interpret the compression as \emph{keeping} some of them (though we only keep the locations, not the values).
In the case of $C_5$, when we receive a sample that contains $3$ identical values, then we call them a \emph{triple} 0 or a \emph{triple} 1, depending on the value.
Recall that a triple 1 can only occur at 3 consecutive positions, and a triple 0 can only occur at 3 non-consecutive positions, so the set of positions determines whether it is a triple 0 or a triple 1.
%

\begin{defi}
	The \emph{join} of two families of functions $\F*\G=\{(f,g)\mid f\in \F, g\in \G\}$ is a family over the disjoiont union of there domains where $(f,g)(x)=f(x)$ if $x$ belongs to the domain of \F and $g(x)$ if $x$ belongs to the domain of \G. 
	When we take the join of several copies of the same family, we use the notation
	$\F^{*n}=\underbrace{\F*\ldots*\F}_{n\textit{ times}}.$
\end{defi}

We obviously have $\vc(\F*\G)=\vc(\F)+\vc(\G)$, but for compression schemes only $\ucs(\F*\G)\le \ucs(\F)+\ucs(\G)$ follows from the definition, and equality does not always hold, as the following statement shows. Recall that $\ucs(C_5)=3$ by Theorem~\ref{thm}.

\begin{prop}\label{prop:2C5}
	$\ucs(C_5*C_5)\le 5$.
\end{prop} 
\begin{proof}
\begin{table}[h]
	\centering
	\begin{tabular}{|c|c|c|}
		\hline
		\textit{\hspace{5mm}Sample\hspace{5mm}} & \hspace{5mm}\textit{Compression}\hspace{5mm} & \hspace{5mm}\textit{Decoding}\hspace{5mm} \\ \thickhline
		{no triples}      & {keep all 1's}   & {kept to 1, rest 0} \\ \hline
		triple 1 in  $C_5^{(1)}$                        & keep triple and 1's in $C_5^{(2)}$                 & triple from position,\\ \cline{1-2}
		triple 0 in  $C_5^{(1)}$                        & keep triple and 0's in $C_5^{(2)}$                 & kept in $C_5^{(2)}$ same\\ \hline
		triple 1 in  $C_5^{(2)}$                        & keep triple and 0's in  $C_5^{(1)}$                 & triple from position,\\ \cline{1-2}
		triple 0 in $C_5^{(2)}$                       & keep triple and 1's in  $C_5^{(1)}$                 & kept in  $C_5^{(1)}$ opposite\\ \hline
	\end{tabular}
\caption{Compressing $C_5*C_5$.}\label{tab:2C5}
\end{table}
For the proof we need to give an unlabeled compression scheme $(\alpha,\beta)$.
	There are several possible schemes, one is sketched in Table \ref{tab:2C5}.
	The compression $\alpha$ depends on whether there are, and what type of triples in the labeled sample restricted to the domains of the two copies of $C_5$. We denote these domains by $C_5^{(1)}$  and $C_5^{(2)}$.

	If neither of them contains a triple, we just keep the 1's in the labeled sample.
	
	If  $C_5^{(1)}$ contains a triple 1, but $C_5^{(2)}$ does not contain a triple 1, then we still just keep the 1's.

	If  $C_5^{(1)}$ contains a triple 0, but $C_5^{(2)}$ does not contain a triple 0, then we keep all the 0's in the labeled sample.
	
	If  $C_5^{(2)}$ contains a triple 1, but $C_5^{(1)}$ does not contain a triple 0, then keep the triple 1 from $C_5^{(2)}$, and the 0's from $C_5^{(1)}$.

	If  $C_5^{(2)}$ contains a triple 0, but $C_5^{(1)}$ does not contain a triple 1, then keep the triple 0 from $C_5^{(2)}$, and the 1's from $C_5^{(1)}$.

	Note that if the compressed sample contains three positions from either $C_5^{(1)}$ or $C_5^{(2)}$, then those positions formed a triple in the labeled sample and it was a triple 1 in case of three consecutive positions and a triple 0 in case of three non-consecutive positions. This means that the compressed sample determines which one of the five rules was used to obtain it and the decoding $\beta$ can be constructed accordingly.

	Finally, notice that exactly one of the above 5 cases happens for every sample.
	(Although note that for us it would be sufficient if \emph{at least} one of them happened for every sample.)
\end{proof}

This raises the question of how $\ucs(\F^{*n})$ behaves when $n\to \infty$.
We can prove neither any lower bound that would be better than $n\cdot\vc(\F)$ for any $\F$ at all (notice that Proposition \ref{prop:2C5} only provides an upper bound, but we do not know whether in general $\ucs(\F*\G) \ge \ucs(\F)+\ucs(\G)-1$ holds or not), nor show that $\ucs(\F^{*n})\le (1+o(1))n\cdot\vc(\F)$ for every $\F$.
We make the following conjecture.

\begin{conj}
	$\lim_{n\to \infty} \frac{\ucs(C_5^{*n})}n$ exists and is strictly larger than $2$.
\end{conj}


We can prove that $\ucs(C_5^{*n})\le 2n+1$ for $n\le 5$.
Since the compression schemes are based on similar ideas, we only sketch the scheme for $n=5$.

\begin{prop}\label{prop:5C5}
	$\ucs(C_5^{*5})\le 11.$
\end{prop} 
\begin{proof}
	\begin{table}[h]
		\centering
		\begin{tabular}{|m{6cm}|m{6cm}|}
			\hline
			\textit{\hspace{20mm}Sample\hspace{20mm}} & \hspace{20mm}\textit{Compression}\hspace{20mm} \\ \thickhline
			{no triple 1}      & {keep all 1's}   \\ \hline
			
			triple 1 in some $C_5^{(i)}$\newline
			 but no triple 0 anywhere     & keep triple 1 in  $C_5^{(i)}$\newline and 0's in other  $C_5^{(j)}$'s     \\ \cline{1-2}
			
			exactly one triple 0   & keep  0's \\ \hline
			
			exactly one triple 1\newline and at least two triple 0's  & fix two triple 0's and one triple 1; keep non-central triples and central element of central triple, and 1's from rest \\ \hline
			
			at least two triple 1's and at least two triple 0's, and fifth does not have exactly one 1  & keep triple 1's and central elements of triple 0's, and 1's from fifth \\ \hline
			
			two triple 1's and at least two triple 0's, and fifth has exactly one 1  & keep triple 1's and non-central elements of triple 0's, and 1 from fifth  \\ \hline
			
		\end{tabular}
		\caption{Compressing $C_5^{*5}$.}\label{tab:5C5}
	\end{table}
We denote the $5$ copies of $C_5$'s by $C_5^{(0)},\ldots,C_5^{(4)}$, with indexing $\bmod~5$.

Among any three positions in a single $C_5^{(i)}$ there is a unique ``central'' element: the one that is equidistant from the other two elements. We use that the two non-central elements determine the central element uniquely. Although the central element is not enough to determine the other two elements, it becomes enough once we know whether they are the positions in a triple 0 or a triple 1.

Similarly, among any three distinct sets $C_5^{(i)}$, $C_5^{(j)}$ and $C_5^{(k)}$, there is a unique central one, whose index is equidistant (modulo 5) from the other two indices.
E.g., from $C_5^{(0)}$, $C_5^{(2)}$ and $C_5^{(3)}$ the central one is $C_5^{(0)}$, while from $C_5^{(0)}$, $C_5^{(3)}$ and $C_5^{(4)}$ the central one is $C_5^{(4)}$. We use again that the non-central copies determine the central one uniquely.

The compression algorithm is sketched in Table \ref{tab:5C5}.
This Table needs to be interpreted in a similar fashion as Table \ref{tab:2C5}, this time we omit the lengthy description of the case analysis.
Note that for some labeled samples there are more rules to choose from for the compression -- in this case, we pick arbitrarily. It is important, however that there is always at least one rule that applies.

We have also omitted the decompression rules, as the compressed sample always determines which rule was used to obtain it.
To prove this statement, notice that we only keep three position of the same $C_5^{(i)}$ if they form a triple in the labeled sample. If the first rule is used, no triple is kept. In case the second or third rule is used, a single triple 1 or triple 0 is kept, respectively. If the fourth rule is used, then two triples are kept, not both triple 1's. Finally if either of the last two rules are used, then at least two triple 1's are kept. The compressed sample produced by the last two rules are distinguished by the number of elements kept in the sets $C_5^{(i)}$: if it is $3+3+2+2+1$ in some order, then the last rule was used, otherwise the fifth rule. Once we know which rule produced the compressed sample the decoding can be done accordingly.
\end{proof}

\section{Further results}

In this section we mention some further results. 
We start by defining some further families.

$C_5^-$ is obtained from $C_5$ by deleting one function.
Because of the symmetry, it does not matter which one, so we delete the function 0-1-1-1-0. Here we represent functions by the sequence of their values on 0, 1, 2, 3, 4.
In this family, still any two positions can take any values ($4$ possibilities each), but for some triples we have only $6$ possibilities (instead of $7$).

$C_4$ is the restriction of $C_5$ to four elements of the domain.
Again, by symmetry it does not matter which four, so we delete the central element 2.
This is useful, because this way $C_4$ also becomes a restriction of $C_5^-$.

\begin{prop}\label{prop:C5-}
	$\ucs(C_4)=\ucs(C_5^-)=2$.
\end{prop}
\begin{proof}
	\begin{table}[h]
		\centering
		\begin{tabular}{|c|c|c|}
			\hline
			\textit{\hspace{20mm}Compression\hspace{20mm}} & \hspace{20mm}\textit{Decoding}\hspace{20mm} \\ \thickhline
			 $\emptyset$ & 1-0-0-0-1      \\ \hline
			 x-.-.-.-.   & 0-0-1-0-1      \\ \hline
			 .-x-.-.-.   & 1-1-0-0-1      \\ \hline
			 .-.-x-.-.   & 1-0-1-0-1      \\ \hline
			 x-x-.-.-.   & 0-1-0-0-1      \\ \hline
			 x-.-x-.-.   & 0-1-0-0-1      \\ \hline
			 x-.-.-x-.   & 0-0-1-1-1      \\ \hline
			 x-.-.-.-x   & 0-1-0-1-0      \\ \hline
			 .-x-x-.-.   & 1-1-1-0-0      \\ \hline
			 .-x-.-x-.   & 0-1-0-1-0      \\ \hline	
		\end{tabular}
		\caption{Compressing $C_5^-$; elements of the compressed sample are marked with an x.}\label{tab:C5-}
	\end{table}
	The lower bounds follow from $2=\vc(C_4)\le \ucs(C_4)\le \ucs(C_5^-)$.
	For the upper bound, we need to give a compression scheme of size two for $C_5^-$.
	A possible algorithm is sketched in Table \ref{tab:C5-}.
	Here we list the decoding of compressed samples only. We maintain a symmetry for the reflection to the central element: If the compressed sample $B$ is obtained from another compressed sample $A$ by reflection, then the decoding $\beta(B)$ is also obtained from $\beta(A)$ the same way. Accordingly, we only list one of $A$ and $B$ in the Table.
	We omit the lengthy case analysis of why this compression scheme works.	
%
%
%
%
%
\end{proof}

Now we continue by definining two more families.\\

$P(k)$ is the family of all $2^k$ boolean functions on a domain of $k$ elements. Notice that $P(k)=P(1)^{*k}$.
As $P(k)$ shatters its entire domain, we have $\vc(P(k))=k$. We also have $\ucs(P(k))=k$ as $\vc(P(k))\le \ucs(P(k))$ and $\ucs(P(k)\le k$ is shown by the simple unlabeled compression scheme that keeps the 1's in the labeled sample.
On the other hand, $\lcs(P(k))$ can be smaller, e.g., $\lcs(P(2))=1$.

$W_6$ is a symmetrizing extension of $C_5$, with the same number of functions, but one more base element. One can obtain it from $C_5$ by adding an extra element to the base and extending each function in the family to the new element such that the function has three zeros and three ones.
Figure \ref{fig:W6} depict two functions of $W_6$. The other eight functions are the rotations of these two. In the family $W_6$ the extra element plays no special role, in fact, $W_6$ is two-transitive, i.e., any pair of elements of its domain can be mapped to any other pair of elements with an automorphism. If we convert the functions of $W_6$ to $3$-element sets, we get the unique $2-(6,3,2)$ design. 
Since $W_6$ is an extension of $C_5$, $\vc(C_5)\le \vc(W_6)$ and $\ucs(C_5)\le \ucs(W_6)$ -- it is easy to check that we have equality in both cases, i.e., $\vc(W_6)=2$ and $\ucs(W_6)=3$.

\begin{figure}[h]
	\begin{center}
		\includegraphics[width=5cm]{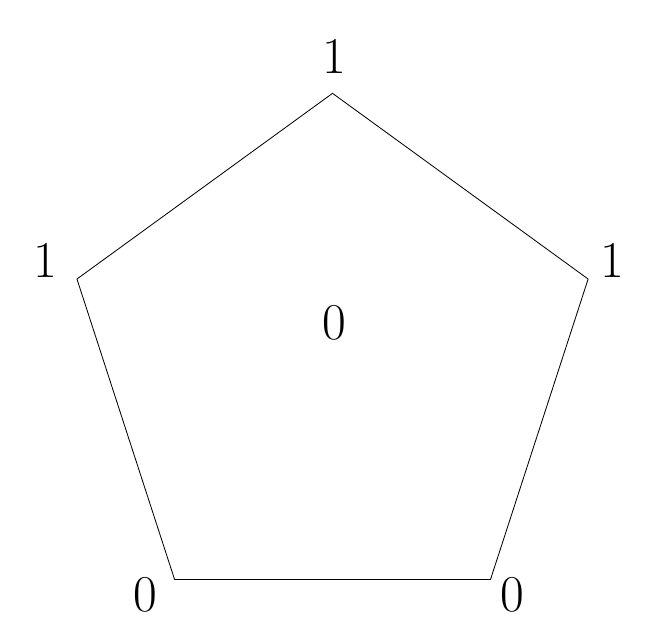}~~~~~~~~~~~
		\includegraphics[width=5cm]{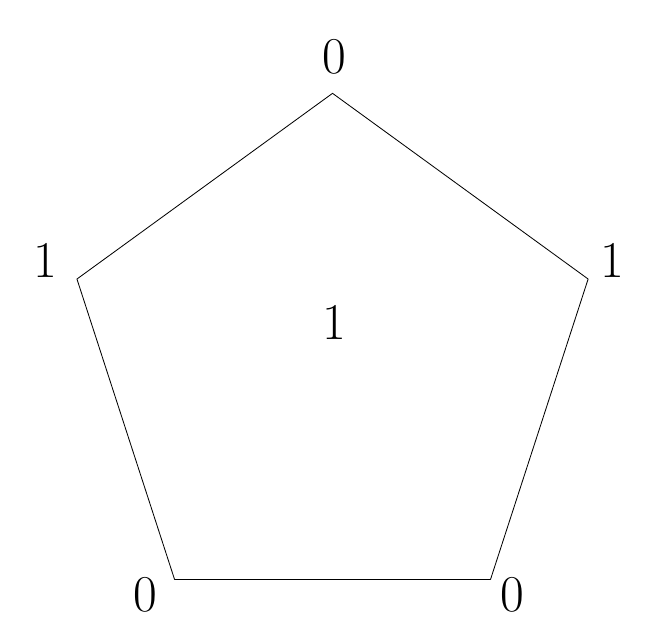}%
	\end{center}
	\caption{$W_6$ consists of the $5$ rotations of the above sets.}
	\label{fig:W6}
\end{figure}

Some further non-trivial upper bounds can be obtained for the joins involving these families.

\begin{prop}\label{prop:W6P1}
$\ucs(W_6* P(1))=3$.
\end{prop}
\begin{proof}
		\begin{table}[h]
		\centering
		\begin{tabular}{|c|c|c|}
			\hline
			\textit{\hspace{5mm}Sample\hspace{5mm}} & \hspace{5mm}\textit{Compression}\hspace{5mm} & \hspace{5mm}\textit{Decoding}\hspace{5mm} \\ \thickhline
			extra is not 1          & keep 1's of $W_6$       & kept 1, others 0 \\ \hline
			extra is 1 and triple 0 & keep triple 0        & kept 0, others 1 \\ \hline
			extra is 1, no triple 0 & keep extra and 0's     & extra 1, rest of kept 0, others 1 \\ \hline

		\end{tabular}
		\caption{Compressing $W_6* P(1)$.}\label{tab:W6P1}
	\end{table}
	The compression algorithm is sketched in Table \ref{tab:W6P1}, with `extra' denoting the only bit of the domain of $P(1)$.
\end{proof}

Note that $C_5*P(1)$ is obtained from $W_6*P(1)$ by restricting the domain and such a restriction cannot increase the value of $\ucs$, so this also implies $\ucs(C_5* P(1))=3$. From this we can easily get another proof for $\ucs(C_5* C_5)\le 5$ as follows.
We have $C_5\subset P(1)*C_4$, thus $\ucs(C_5*C_5)\le\ucs(C_5*P(1)*C_4)\le\ucs(C_5*P(1))+\ucs(C_4)\le3+2$, using Proposition \ref{prop:C5-}.

\begin{prop}
	$\ucs(W_6* W_6)\le 5$.
\end{prop}
\begin{proof}
This compression goes similarly to the one presented in Table \ref{tab:2C5} for $C_5 * C_5$.
In fact, we can use exactly the same compression scheme unless we get two triples in both $W_6$'s, i.e., a labeled sample that contains all $12$ elements of the base.
There are $10\cdot 10=100$ possibilities for such a sample, and for each we can pick a compression that keeps at least $4$ elements from at least one of the two copies of $W_6$, as these were not used yet.
There are $\binom 65\cdot \binom 60+\binom 64\cdot \binom 61+\binom64\cdot\binom60+\binom 60\cdot\binom64+\binom61\cdot \binom 64+\binom 60\cdot \binom 65=222$ such possible compressed samples, we can use a distinct one for each of the $100$ problematic labeled samples. This makes the decoding possible.
\end{proof}

%

We end by a summary of the most important questions left open.

\subsubsection*{Summary of main open questions}

\begin{itemize}
	
\item Is $\ucs(\F)-\vc(\F)$ bounded?

\item Is $\ucs(\F*\G) \ge \ucs(\F)+\ucs(\G)-1$? 

\item How does $\ucs(C_5^{*n})$ behave? Does $\lim \ucs(n * \F)/n$ exist?

\item Is there a $k$ for every \F such that $\ucs(\F * P(k))=\vc(\F)+k$? 

\end{itemize} 

%

\subsection*{Remarks and acknowledgment}
We would like to thank Tam\'as M\'esz\'aros, Shay Moran and Manfred Warmuth for useful discussions and calling our attention to new developments.

\end{document}